\documentclass{amsart}

\usepackage[ruled,vlined, algo2e]{algorithm2e}

\usepackage[english, activeacute]{babel}
\usepackage[latin1]{inputenc}
\usepackage{hyperref}
\usepackage{graphicx}
\usepackage[left=2.7cm, right=2.7cm, top=2.5cm, bottom=2.7cm]{geometry}

\usepackage{amsmath}
\usepackage{amssymb}
\usepackage{latexsym}

\newcommand{\SG}{\mathrm{SG}}
\newcommand{\g}{\mathrm{g}}
\newcommand{\F}{\mathrm{F}}
\newcommand{\G}{\mathrm{G}}
\newcommand{\m}{\mathrm{m}}
\newcommand{\Ap}{\mathrm{Ap}}
\newcommand{\Q}{\mathbb{Q}}

\newcommand{\N}{\mathbb{N}}

\newcommand{\K}{\mathbb{K}}

\newcommand{\dsum}{\displaystyle\sum}

\newtheorem{theo}{Theorem}
\newtheorem{defi}[theo]{Definition}
\newtheorem{ex}[theo]{Example}

\newtheorem{cor}[theo]{Corollary}
\newtheorem{prop}[theo]{Proposition}

\newtheorem{lemma}[theo]{Lemma}

\title{Irreducible numerical semigroups with multiplicity three and four}
\author{V\'ictor Blanco}

\address{Departamento de \'Algebra, Universidad de Granada}
\email{vblanco@ugr.es}
\date{\today}

\keywords{Numerical Semigroups, Irreducibility, Multiplicity, Ap\'ery set}
\subjclass[2010]{20M14, 11P21, 11D75, 13H10}

\begin{document}
\maketitle
\begin{abstract}
In this paper we analyze the irreducibility of numerical semigroups with multiplicity up to four. Our approach uses the notion of Kunz-coordinates vector of a numerical semigroup recently introduced in \cite{siam11}. With this tool we also completely describe the whole family of minimal decompositions into irreducible numerical semigroups with the same multiplicity for this set of numerical semigroups.
We give detailed examples to show the applicability of the methodology and conditions for the irreducibility of well-known families of numerical semigroups as those that are generated by a generalized arithmetic
progression.
\end{abstract}

\section{Introduction}
\label{sec:0}
A numerical semigroup is a subset $S$ of $\N$ (here $\N$ denotes the set of non-negative integers) closed under
addition, containing zero and such that $\N\backslash S$ is finite.
 Numerical semigroups were first considered while studying the
set of nonnegative solutions of Diophantine equations and their study is closely related to the analysis of monomial curves (see \cite{delorme}).
By these reasons, the theory of numerical semigroups has attracted a number of researchers from the algebraic community. For instance, some terminology from algebraic geometry has been exported to this field as the multiplicity (the smallest positive integer belonging to the semigroup), the genus (the number of nonnegative integers not belonging to the semigroup), or the embedding dimension (the cardinal of the minimal system of generators of the semigroup). Further details about the theory of numerical semigroups can be found in the recent monograph by Rosales and Garc\'ia-S\'anchez \cite{springer}.

A numerical semigroup is said irreducible if it cannot be expressed as an intersection of two numerical semigroups containing it properly. This notion was introduced in \cite{rosales03} where it is also shown that the family of irreducible numerical semigroups is the union of two families of numerical semigroups with special importance in this theory: symmetric and pseudo-symmetric numerical semigroups. The Frobenius number of a numerical semigroup is the largest integer not belonging to the semigroup. Then, symmetric (resp. pseudo-symmetric) numerical are those irreducible numerical semigroups with odd (resp. even) Frobenius number (see \cite{barucci97,froberg87}).

The irreducibility of a numerical semigroup have been widely studied in the literature (see \cite{branco-nuno07,rosales02,rosales02b,rosales03,rosales03b,rosales04}).
Furthermore, apart from the theory of semigroups, this notion is connected with commutative ring theory. In fact, let $S$ be a numerical semigroup, $\K$ a field, and
$\K[[t]]$ the ring of formal power series over $\K$. It is well-known  that $\K[[S]]=\{\dsum_{s\in S} a_s t^s: a_s \in \K\}$ is a subring of $\K[[t]]$, called the ring
of the semigroup associated to $S$ (see for instance, \cite{barucci97}). Properties over the numerical semigroup $S$ are translated to the ring associated to $S$. Actually, it is well-known that
 if the numerical semigroup is symmetric, the ring associated to it is a Gorenstein ring (see \cite{kunz1}) and if the semigroup is pseudo-symmetric, the ring is a Kunz ring (see \cite{barucci-froberg}).

In \cite{ijac11}, it is introduced the notion of $m$-irreducibility, which extends the concept of irreducibility when the multiplicity is fixed. A numerical semigroup with multiplicity $m$ is said $m$-irreducible if it cannot be expressed  as an intersection of two numerical semigroups with multiplicity  $m$ and containing it properly. In \cite{ijac11} apart from introducing this notion, the set of $m$-irreducible numerical semigroups is characterized in terms of its special gaps and also by its genus and Frobenius number. An interesting problem
when treating the irreducibility of a numerical semigroup is to minimally decompose a numerical semigroup into ($m$-)irreducible ones, in the sense that the decomposition involves the minimum number of semigroups.
 In \cite{ijac11} it is also given an algorithm to compute such a decomposition. A different approach, by applying integer programming tools, is proposed in \cite{siam11} to compute  more efficiently (in polynomial time) minimal decompositions of numerical semigroups with multiplicity
$m$, into $m$-irreducible numerical semigroups. In that approach it is used the notion of Kunz-coordinates vector
to translate the considered problem in the problem of finding some integer optimal solutions, with respect to appropriate objective
 functions, in a Kunz polytope. The encoding of a numerical semigroup as a integer vector was first considered in \cite{kunz} and \cite{london02} and it is based on the Ap\'ery set codification of a semigroup with respect to its multiplicity. This useful tool has been also applied to compute the number of numerical semigroups with a given genus \cite{counting}.

Here, we analyze the irreducibility of the family of numerical semigroups with multiplicities $3$ and $4$. The characterizations of this set of numerical semigroups in terms of the Frobenius number
,the genus or the ratio  is studied in \cite{rosales05}. Note that the case when the multiplicity of the semigroup is $2$ is trivial since any numerical semigroup with multiplicity $2$ is symmetric (see \cite{rosales05}) and then irreducible. Although in general the notions of irreducibility and $m$-irreducibility are different, we prove
that these notions coincide when $m$ is three or four (and obviously for $m=2$). Furthermore, we give explicit and simple conditions over the Kunz-coordinates vector of a numerical semigroups to be irreducible in these cases. Then, for a given numerical semigroup with multiplicity $3$ (resp. $4$), we describe minimal decompositions into $3$-irreducible (resp. $4$-irreducible) numerical semigroups.

We apply this approach to analyze some subfamilies of numerical semigroups with multiplicities $3$ or $4$: $3$ and $4$-symmetric numerical semigroups, $3$ and $4$-pseudosymmetric numerical semigroups or semigroups generated by generalized arithmetic sequences.

In Section \ref{sec:1} we recall the main definitions and results needed through this paper. Section \ref{sec:2}
 is devoted to analyze the family of numerical semigroups with multiplicity $3$. We characterize in this section the set of $3$-irreducible numerical semigroups in terms of its
 Kunz-coordinates vector, and we explicitly describe the minimal decomposition of any numerical semigroup with multiplicity $3$ into irreducible numerical semigroups. An analogous analysis is done in Section \ref{sec:3} for the case when the multiplicity is four.

\section{Preliminaries}
\label{sec:1}
A numerical semigroup is a subset $S$ of $\N$ closed under
addition, containing zero and such that $\N\backslash S$ is finite. The reader is referred to the recent monograph by Rosales and Garc\'ia-S\'anchez \cite{springer}
for further details about the theory of numerical semigroups.

The multiplicity of a numerical semigroup $S$ is the smallest non zero element belonging to it, and it is usually denoted by $\m(S)$. A numerical semigroup
 $S$ is said irreducible (resp. $m$-irreducible) if it cannot be expressed as an intersection of two numerical semigroups (resp. numerical semigroups with
 multiplicity $m$) containing it properly. In \cite{ijac11} the authors characterize the set of $m$-irreducible numerical semigroups for any positive integer $m$.
For the sake of completeness, we recall some of these characterizations that will be useful for the development done in this paper.

For any numerical semigroup $S$, the Frobenius number of $S$, $\F(S)$, is the largest integer not belonging to $S$, and the genus of
 $S$, $\g(S)$, is the number of nonnegative integers that do not belong to $S$. The following result completely determines the set of
$m$-irreducible numerical semigroups. Here $\lceil q \rceil$ stands for the ceiling integer part of any $q \in \Q$.

\begin{lemma}[Proposition 6 in \cite{ijac11}]
\label{lemma:1}
A numerical semigroup, $S$, with multiplicity $m$ is $m$-irreducible if and only if one of the following conditions holds:
\begin{enumerate}
\item $S=\{x \in \N: x\ge m\} \cup \{0\}$.
\item $S=\{x \in \N: x\ge m, x\neq \F(S)\} \cup \{0\}$.
\item $S$ is an irreducible numerical semigroup.
\end{enumerate}
\end{lemma}

From the above result is easy to obtain the next lemma.

\begin{lemma}
 \label{lemma:2}
		Let $S$ be a numerical semigroup with multiplicity $m$. Then, $S$ is $m$-irreducible if and only if $\g(S) = \left\{m-1, m,
\left\lceil \dfrac{\F(S)+1}{2} \right\rceil\right\}$.
\end{lemma}

Other useful set that appear when one analyzes irreducible numerical semigroups is the set of special gaps of a semigroup. Let $S$ be a numerical semigroup, the set of special gaps of $S$ is:
$$
\SG(S)=\{h \in \N\backslash S: S  \cup \{h\} \text{ is a numerical semigroup}\}
$$
If $\m(S)=m$, we denote by $\SG_m(S)=\{ h \in \SG(S): h > m\}$ the set of special gaps of $S$ larger than the multiplicity.
 Note that when $h \in \SG_m(S)$, $S'=S \cup \{h\}$ is a numerical semigroup with multiplicity $m$ and that $\F(S) \in SG_m(S)$ if $S \neq \{0, m, \rightarrow\}$.

Then, a $m$-irreducible numerical semigroup can be detected by counting the special gaps larger than $m$. We denote by $\{n_1, \ldots, n_k, \rightarrow\} = \{n_1, \ldots, n_k\} \cup \{n \in \N: n \ge n_k+1\}$ for any $n_1, \ldots, n_k \in \N$.
\begin{lemma}[\cite{ijac11}]
\label{lemma:3}
Let $S$ be a numerical semigroup with multiplicity $m$. Then, $S$ is $m$-irreducible if and only if $\#\SG_m(S) \leq 1$. Furthermore, $\SG_m(S)=\emptyset$ if and only if $S=\{0, m, \rightarrow\}$.
\end{lemma}

Once the set of $m$-irreducible numerical semigroups is characterized, one may be interested in decomposing a numerical semigroup as an intersection of $m$-irreducible numerical semigroups. Actually, any numerical semigroup with multiplicity $m$ can be decomposed into $m$-irreducible numerical semigroups (\cite[Proposition 1]{ijac11}). In that paper it is also given an algorithm for obtaining a minimal (in the sense of the minimum number of elements involved in the intersection) decomposition into $m$-irreducible numerical semigroup by using the Ap\'ery set of the numerical semigroup. In \cite{siam11}, some algorithms for obtaining such a minimal decomposition by formulating the problem as an integer programming model and by using the notion of Kunz-coordinated vector of a numerical semigroup.

\begin{defi}[Ap\'ery set and Kunz-coordinates]
\label{def:5}
Let $S$ be a numerical semigroup and $s \in S$.
\begin{enumerate}
\item The \emph{Ap\'ery set} of $S$ with respect to $s \in S$ is the set $\Ap(S,s) = \{w_0=0, w_1, \ldots, w_{s-1}\}$, where $w_i$ is the least
element in $S$ congruent with $i$ modulo $s$, for $i=1, \ldots, s-1$.
\item The \emph{Kunz-coordinates vector} of $S$ is the integer vector $x \in \N^{m-1}$ whose components are $x_i=\frac{w_i-i}{m}$ where $m=\m(S)$ and $\Ap(S, m)=\{w_0=0, w_1, \ldots, w_{m-1}\}$.
\end{enumerate}
\end{defi}
The Kunz-coordinates vector were implicity used in \cite{london02} and in \cite{counting} for counting numerical semigroups of a given genus. The importance of the Kunz-coordinates vector for representing a numerical semigroup is also shown in the following result.
\begin{lemma}[Theorem 11 in \cite{london02}]
\label{lemma:6}
Each numerical semigroup is one-to-one identified with its Kunz-coordinates.

Furthermore, the set of Kunz-coordinates vectors of the numerical semigroups with multiplicity $m$ is the set of solutions of the following system of diophantine inequalities:
\begin{align}
x_i  \geqslant&1 & \mbox{for all $i \in \{1, \ldots, m-1\}$,}\nonumber\\
x_i+x_j-x_{i+j} \geqslant& 0 & \mbox{for all $1 \leqslant i \leqslant j \leqslant m-1$, $i+j \leqslant m-1$,}\label{kunz}\\
x_i+x_j-x_{i+j-m}  \geqslant& -1 &\mbox{for all $1 \leqslant i \leqslant j \leqslant m-1$, $i+j > m$},\nonumber\\
x_i \in \N & &\mbox{for all $i \in \{1, \ldots, m-1\}$}\nonumber.
\end{align}
\end{lemma}
The bijective correspondence given in the above result between numerical semigroups with multiplicity $m$ is given by $(x_1, \ldots, x_{m-1}) \mapsto \langle mx_1+1, \ldots, mx_{m-1}+m-1\rangle$.

An integer vector $x \in \N^{m-1}$ is said a Kunz-coordinates vector if it is the Kunz-coordinates vector of some numerical semigroup with multiplicity $m$. Thus, being a Kunz-coordinates vector is equivalent to be a solution of the diophantine system of inequalities \eqref{kunz}.

The Frobenius number, the genus and the special gaps larger than the multiplicity of a numerical semigroup can be computed by manipulating the Kunz-coordinates vector of a numerical semigroup (see \cite{siam11}): Let $S$ be a numerical semigroup with multiplicity $m$ and $x\in \N^{m-1}$ its Kunz-coordinates vector. Then, by Selmer's formulas \cite{selmer77}, $\g(S)=\sum_{i=1}^{m-1} x_i$, $\F(S)=\max\{mx_i+i\}-m$ and
\begin{equation}
\label{sg}
\begin{array}{lll}
\SG_m(S)&=\{h_i=m(x_i-1)+i :& x_i+x_j > x_{i+j} \text{ for $j$ such that } i+j<m,\\
  & & x_i+x_j > x_{i+j-m}-1 \text{ for $j$ such that  } i+j>m,\\
  & &\text{and } 2h_i \ge mx_{2h_i \pmod m} + {2h_i \pmod m},\\
  & &i=1, \ldots, m-1\}.
  \end{array}
\end{equation}

Hence, by lemmas \ref{lemma:6} and \ref{lemma:2} and the expression of the genus and the Frobenius number of a numerical semigroup in terms of its Kunz-coordinates vector, the set of Kunz-coordinates vectors of all the $m$-irreducible numerical semigroups is the entire set of solutions of system \eqref{kunz} when adding the constraint $\dsum_{i=1}^{m-1} x_i \in  \{m-1, m, \max_i\;\{mx_i+i\}- m\}$.

We are interested in decomposing a numerical semigroup $S$ with multiplicity $m$ into $m$-irreducible numerical semigroups. Then, the $m$-irreducible numerical semigroups involved in such a decomposition can be found among the set of oversemigroups of $S$ with multiplicity $m$, that is, the set
$$
\mathcal{O}_m(S) = \{S' \text{ numerical semigroup}: S \subseteq S' \text{ and } \m(S')=m\}
$$
If $x \in \N^{m-1}$ is the Kunz-coordinates vector of $S$, the set $\mathcal{O}_m(S)$ is one-to-one identified, in terms of Kunz-coordinates vectors, with the set of \emph{undercoordinates} of $x$, that is, with the set
$$
\mathcal{U}_m(x) = \{x' \in \N^{m-1}: x' \text{ is a Kunz-coordinates vector and } x' \leq x\}
$$
where $\leq$ stands for the component-wise order in $\N^{m-1}$ (see \cite{ijac11}).

Thus, the undercoordinates of a Kunz-coordinates vector are in the form $x-y$ with $y\in\N^{m-1}$ and such that $x-y$ is a Kunz-coordinates vector. Since we are interested in decomposing a numerical semigroup $S$ whose Kunz-coordinates vector is $x\in \N^{m-1}$, into $m$-irreducible numerical semigroups, by applying \eqref{kunz} to $x-y$, the conditions for $y$ to be $x-y$ a $m$-irreducible Kunz-coordinates vector are:

\begin{align}
y_i \leqslant x_i - 1 & \mbox{ for all $i \in \{1, \ldots, m-1\}$,}\nonumber\\
y_i + y_j - y_{i+j} \leqslant x_i + x_j - x_{i+j} & \mbox{ for all $1 \leqslant i \leqslant j \leqslant m-1$, $i+j \leqslant m-1$,}\nonumber\\
y_i + y_j - y_{i+j} \leqslant x_i + x_j - x_{i+j} + 1 &\mbox{ for all $1 \leqslant i
\leqslant j \leqslant m-1$, $i+j > m$},\label{polytope:x}\tag{${\rm P}^m(x)$}\\
\dsum_{i=1}^{m-1} y_i \in M(x, y),\label{eq:disj}\\
y \in \N^{m-1}_+.\nonumber
\end{align}
where $M(x, y) = \{\dsum_{i=1}^{m-1} x_i -m, \dsum_{i=1}^{m-1} x_i -m + 1,$ $\dsum_{i=1}^{m-1} x_i -
$ $\left\lceil \frac{\max_i\{m(x_i-y_i) + i\} - m +1}{2} \right\rceil\}$.

If condition \eqref{eq:disj} is $\dsum_{i=1}^{m-1} y_i= \dsum_{i=1}^{m-1} x_i -m + 1,$, the unique solution is $x-y=(1, \ldots, 1)$, whose associated numerical semigroup is $S_m=\{0,m,\rightarrow\}$. $S_m$ is the maximal element in the set of numerical semigroups with multiplicity $m$. Hence, $S_m$ appears only in its own decomposition and in no one else.

By solving the above system of diophantine inequalities, we obtain a decomposition of $S$ into $m$-irreducible numerical semigroups, in terms of its Kunz-coordinates vectors. To get a minimal decomposition the entire set of solutions of \eqref{polytope:x} must be filtered conveniently to avoid redundant solutions. For designing such a filter we use the following result:
\begin{lemma}[\cite{ijac11}]
\label{lemma:7}
Let $S$ be numerical semigroups with multiplicity $m$ and $S_1, \ldots,$ $S_n \in \mathcal{O}_m(S)$. $S=S_{1} \cap \cdots \cap S_{n}$ if and only if $\SG_m(S) \cap \left(\G(S_{1}) \cup \cdots \cup \G(S_{n})\right) = \SG_m(S)$.
\end{lemma}
From the above result, the problem of minimally decomposing a numerical semigroup with multiplicity $m$, $S$, into $m$-irreducible numerical semigroups is translated into the problem of finding a set of oversemigroups of $S$ that minimally covers the special gaps larger than $m$ of $S$. To check if a solution of \eqref{polytope:x} contains an specific special gap $h \in \SG_m(S)$, in \cite{siam11} it is proven the following result that analyzes the structure of the system in terms of the elements in $\SG_m(S)$.
\begin{lemma}
\label{lemma:8}
Let $S$ be a numerical semigroup with multiplicity $m$ and Kun-coordinates vector $x\in \N^{m-1}$. Then, there exists a minimal decomposition of $S$ into $m$-irreducible numerical semigroups $S=S_1 \cap \cdots \cap S_k$ with the following properties:
\begin{enumerate}
\item $h_i = \F(S_i) \in \SG_m(S)$.
\item If $x^i=x-y^i$ is the Kunz-coordinates vector of $S_i$, then $y^i_{h_{i} \pmod m}=0$, for $i=1, \ldots, k$.
\item $\dsum_{j=1}^{m-1} y^i_j = \left\{\begin{array}{rl} \dsum_{i=1}^{m-1} x_i - \left\lceil \dfrac{h_i+1}{2} \right\rceil & \mbox{if $h_i>2m$,}\\
 \dsum_{i=1}^{m-1} x_i -m & \mbox{if $h_i<2m$.}
 \end{array}\right.$, for $i=1, \ldots, k$.
 \end{enumerate}
 \end{lemma}
Finally, we recall this useful result that also appears in \cite{siam11}:
\begin{prop}
\label{prop:9}
Let $x \in \N^{m-1}_+$ be a Kunz-coordinates vector, $y \in \N^{m-1}_+$ and $h \in \SG_m(x)$. If $x-y$ is a
 undercoordinate of $x$, then, $h \in \G(x-y)$ if and only if $y_{h \pmod m} =0$.
Furthermore, $\F(x-y)$ is the unique element in $\{h \in \SG_m{x}: h \pmod m  = \max\{i \in \{1, \ldots, m-1\}: y_i=0\}\}$.
\end{prop}
\section{3-irreducible numerical semigroups}
\label{sec:2}
In this section we analyze the set of $3$-irreducible numerical semigroups, that is, the set of numerical semigroups with multiplicity $3$ and such that cannot be expressed as an intersection of numerical semigroups with multiplicity $3$. Once this set is described, we give explicit decomposition into $3$-irreducible numerical semigroups for any numerical semigroup with multiplicity three.

It is clear that every irreducible numerical semigroup with multiplicity $m$ is also $m$-irreducible. However, in general, the converse is not true. First we prove that both notions are equivalent when the multiplicity is three.
\begin{lemma}
\label{lemma:10}
Every $3$-irreducible numerical semigroup is irreducible.
\end{lemma}
\begin{proof}
Let $S$ be a $3$-irreducible numerical semigroup. By Lemma \ref{lemma:1} one of the following conditions must be satisfied:
\begin{enumerate}
\item $S=\{x \in \N: x\ge 3\} \cup \{0\}$. In this case $S=\langle 3,4,5 \rangle$, that is irreducible ($2=\g(S) = \left\lceil \frac{\F(S)+1}{2}\right\rceil = \left\lceil \frac{2+1}{2}\right\rceil = 2$.
\item $S=\{x \in \N: x\ge m, x\neq \F(S)\} \cup \{0\}$. In this case, either $S=\langle 3, 4\rangle$ or $S=\langle 3, 5, 7\rangle$. Both numerical semigroups are irreducible.
\item $S$ is an irreducible numerical semigroup. In this case we are done.
\end{enumerate}
\end{proof}

From the above lemma the conditions for a numerical semigroup with multiplicity $3$ to be $3$-irreducible are also valid to check if the numerical semigroup is irreducible. Furthermore, as a consequence of Proposition 1 in \cite{ijac11} and the above result, any numerical semigroup with multiplicity $3$ can be minimally decomposed into irreducible numerical semigroups with multiplicity $3$.

Let $S$ be a numerical semigroup with multiplicity $3$. Its Kunz-coordinates vector consists of a positive integer vector with two components $x=(x_1, x_2) \in \N^2$. Then, $S$ is $3$-irreducible if $(x_1, x_2)$ is a solution of any of the two systems described by system \eqref{kunz} when adding the constraint $\dsum_{i=1}^{m-1} x_i \in  \{m-1, m, \max_i\;\{mx_i+i\}- m\}$. Also, by Lemma \ref{lemma:3}, $S$ is $3$-irreducible if its set of special gaps larger than the multiplicity has $0$ or $1$ elements. In the following result we explicitly describe the set of special gaps greater than $3$ of $S$.
\begin{prop}
\label{prop:11}
Let $S$ be a numerical semigroup with multiplicity $3$ with Kunz coordinates vector $x=(x_1, x_2)$. Then, the set of special gaps larger than $3$ is:
$$
\SG_3(S) = \left\{\begin{array}{cl} \{\} & \mbox{if $x=(1,1)$},\\
\{3x_1-2\} & \mbox{if $2x_1 \geq x_2+2$, $x_1\geq 2$ and $2x_2\leq x_1$},\\
\{3x_2-1\} & \mbox{if $2x_2\geq x_1+1$, $x_2\geq 2$ and $2x_1 \leq x_2+1$,}\\
\{3x_1-2, 3x_2-1\} & \mbox{if $2x_1 \geq x_2+2$, $2x_2 \geq x_1+1$ and $x_1, x_2 \geq 2$}.\end{array}\right.
$$
\end{prop}
\begin{proof}
By the description of $\SG_3(S)$ in terms of its Kunz-coordinates vector \eqref{sg}, we only need to check if $3(x_1-1)+1 = 3x_1-2$ and $3(x_2-1)+2=3x_2-1$ are special gaps larger than the multiplicity or not. First, those elements must be greater than $m$, i.e., $x_1 \geq 2$ and $x_2\geq 2$.
\begin{itemize}
\item For $h_1=3x_1-2$, $h_1\in \SG_3(S)$ if and only if $x_1 + x_1 > x_2$ (to be in $M$) and $2\left(3(x_1-2)+1)\right) \ge 3x_2 +2$ since $h_1 \pmod 3 = 1$. Equivalently, if $2x_1 > x_2$ and $2x_1 \ge x_2+2$. Clearly, it is deduced that $x_1\geq 2$.
\item For $h_2=3x_2-1$, $h_2\in \SG_3(S)$ if and only if $x_2 + x_2 > x_1-1$ and $2\left(3(x_2-1)+1)\right) \ge 3x_1 +1$ since $h_2 \pmod 3 = 2$. Equivalently, if $2x_2 > x_1-1$ and $2x_2 \ge x_1+1$. We also have that $x_2\geq 2$.
\end{itemize}
By combining both possibilities we obtain the result.
\end{proof}

From the above result we can completely characterize, in terms of their Kunz.coordinates vectors, those numerical semigroups with multiplicity $3$ that are $3$-irreducible by applying Lemma \ref{lemma:3}.

\begin{theo}
\label{theo:10}
Let $S$ be a numerical semigroup with multiplicity $3$ and Kunz-coordinates vector $x=(x_1, x_2) \in \N^2$. Then, $S$ is irreducible if an only if one of the following conditions holds:
\begin{enumerate}
\item $x_1=x_2=1$,
\item $2x_1 \geq x_2+2$ and $2x_2\leq x_1$.
\item $2x_2 \geq x_1+1$ and $2x_1\leq x_2+1$.
\end{enumerate}
\end{theo}

\begin{ex}[Numerical semigroups generated by a generalized arithmetic sequence]
\label{ex:11}
For $h, d, k$ positive integer such that $k \leq 2$ and $\gcd(d, 3)=1$, the numerical semigroup with multiplicity $3$, $S=\langle 3, 3h+d, 3h+2d, \ldots, 3+kd\rangle$ is said that is generated by a generalized arithmetic sequence. In \cite{ramirezalfonsin} it is proved that $\Ap(S, 3)=\{0, 3h\left\lceil\dfrac{1}{k}\right\rceil + d, 3h\left\lceil\dfrac{2}{k}\right\rceil + 2d\}$. Then, if $d=3D+1$ ($d\equiv 1 \pmod 3$), the Kunz-coordinates vector of $S$ is $x=(h\left\lceil \frac{1}{k}\right\rceil + D, h\left\lceil \frac{2}{k}\right\rceil + 2D)$, and  if $d = 3D+2$ ($d \equiv 2 \pmod 3$), $x=(h\left\lceil \frac{2}{k}\right\rceil + 2D+1, h\left\lceil \frac{1}{k}\right\rceil + D)$ otherwise. Then, $S$ is irreducible, by applying Theorem \ref{theo:10}, if and only if one of the following condition holds:
\begin{itemize}
\item $k=1$, ($S=\langle 3, 3h+d\rangle$).
\item $k= 2$ and $h=1$ (In this case, $S=\langle 3, 3+d, 3+2d \rangle$, in  particular, for $d=1$, $S=\{0, 3, \rightarrow\}$).
\end{itemize}
Furthermore, if $k=1$, by Selmer's formulas $\F(S)=6h+2d-3$ which is odd, so $S=\langle 3, 3h+d\rangle$ is symmetric. If $k=2$ and $h=1$, $\F(S)=2d$, so $S$ is always pesudosymmetric in this case.

This result has been previously proven by Matthews in \cite{mathews04} and partially by Estrada and L\'opez in \cite{estrada94}.
\end{ex}

In what follows we show how to decompose a numerical semigroup with multiplicity $3$ that is not irreducible into irreducible numerical semigroups with multiplicity three (the decomposition of a irreducible numerical semigroup into irreducible numerical semigroups is trivial). Assume that $S$ is a numerical semigroup that is not irreducible. By Lemma \ref{lemma:3}, $\#\SG_3(S)=2$, and then, if $x=(x_1, x_2) \in \N^2$ is the Kunz-coordintes vector of $S$, by Proposition \ref{prop:9}, $\SG_3(S)=\{3x_1-2, 3x_2-1\}$ with $2x_1 \geq x_2+2$, $2x_2 \geq x_1+1$ and $x_1, x_2 \geq 2$.

First, we characterize the set of $3$-irreducible oversemigroups of $S$ in terms of their Kunz-coordinates vectors. We denote here by $\mathcal{I}_3(S)$ the set of undercoordinates of the Kunz-coordinates vector of $S$ that are $3$-irreducible.

\begin{theo}
\label{theo:12}
Let $S$ be a numerical semigroup with multiplicity $3$ and such that $S$ is not irreducible. If $x=(x_1, x_2) \in \N^2$ is the Kunz-coordinates vector of $S$, then, the set of irreducible undercoordinates of $S$ with multiplicity $3$ is:
$$
\mathcal{I}_3(S) = \widehat{\mathcal{I}}_3(S) \cup \left\{\begin{array}{rl}
\{(1,1), (x_1, \frac{x_1-1}{2}), (\frac{x_2}{2}, x_2)\} & \mbox{ if $x_1$ is odd and $x_2$ is even,}\\
\{(1,1), (x_1, \frac{x_1-1}{2}), (\frac{x_2+1}{2}, x_2)\} & \mbox{ if $x_1, x_2$ are odd},\\
\{(1,1), (x_1, \frac{x_1}{2}), (\frac{x_2}{2}, x_2)\} & \mbox{ if $x_1, x_2$ are even,}\\
\{(1,1), (x_1, \frac{x_1}{2}), (\frac{x_2+1}{2}, x_2)\} & \mbox{ if $x_1$ is even and $x_2$ is odd,}
\end{array}\right.
$$
where $\widehat{\mathcal{I}}_3(S) = \left\{\begin{array}{rl} \{\{(2,1)\} & \mbox{if $x_1 \geq 2$, $x_2=1$,}\\
\{(1,2)\} & \mbox{if $x_2 \geq 2$, $x_1=1$,}\\
\{(1,2), (2,1) & \mbox{if $x_1, x_2 \geq 2$,}
\end{array}\right.$.
\end{theo}
\begin{proof}
Let $S'$ be a irreducible oversemigroup of $S$. Then, it has Kunz-coordinates vector $(x_1', x_2') =(x_1, x_2)-(y_1, y_2)$  verifying the diophantine inequalities  in \eqref{polytope:x}. Then, for $m=3$, this system is:
\begin{align*}
2y_1-y_2 &\leq 2x_1-x_2\\
2y_2-y_1 &\leq 2x_2-x_1+1
\end{align*}
Let $S'$ be an irreducible oversemigroup of $S$ with multiplicity $3$. It is clear that at least one special gap in $\SG_3(S)$ does not belong to $S'$, otherwise $S'$ has the same special gaps larger than $3$ that $S$, being $S=S'$. We can distinguish here two cases:
\begin{enumerate}
\item $h_1=3x_1-2 \not\in S'$. By Proposition \ref{prop:9}, $h_1 \not\in S'$ if and only if $y_1=0$, and since $S'$ is $3$-irreducible $y_1+y_2 = x_1+x_2 - \left\lceil\frac{3x_1-1}{2}\right\rceil$, that is, $y_2 = x_1+x_2 - \left\lceil\frac{3x_1-1}{2}\right\rceil$.
    \item $h_2=3x_2-1 \not\in S'$. By Proposition \ref{prop:9}, $h_2 \not\in S'$ if and only if $y_2=0$, and since $S'$ is $3$-irreducible $y_1+y_2 = x_1+x_2 - \left\lceil\frac{3x_2}{2}\right\rceil$, that is, $y_1 = x_1+x_2 - \left\lceil\frac{3x_2}{2}\right\rceil$.
\end{enumerate}

Also, if $x_1 \geq 2$, $(2,1) \leq (x_1, x_2)$ is an irreducible undercoordinate of $S$, and if $x_2 \geq 2$, $(1,2) \leq (x_1,x_2)$ is in $\mathcal{I}_3(S)$.

Then, the set of irreducible oversemigroups of $S$ is given by the Kunz-coordinates vectors $(x_1, x_1+ \left\lceil\frac{3x_1-1}{2}\right\rceil)$, $(x_2 + \left\lceil\frac{3x_2}{2}\right\rceil, x_2)$, $(1,2)$, and $(2,1)$.

The result follows by expanding the ceiling par of these vectors and by adding the Kunz-coordinates vector $(1,1)$ which is a irreducible oversemigroup for any numerical semigroup with multiplicity $3$.
\end{proof}

We illustrate the usage of the above result in the following example.

\begin{ex}
Let $S = \langle 3, 10, 14 \rangle$. $S$ is a numerical semigroup with multiplicity $3$ and $\Ap(S, 3)=\{0, 10, 14\}$. Then, its Kunz-coordinates vector is $x=(\frac{10-1}{3}, \frac{4-1}{3})=(3,4)$.
Since $3$ is odd, $4$ is even, and $3, 4 \geq 2$, the set of irreducible undercoordinates of $S$ is $\{(1,2), (2,1)\} \cup \{(1,1), (3, \frac{3-1}{2}), (\frac{4}{2}, 4)\}=\{(1,1), (1,2), (2,1), (3,1), (2,4)\}$. Consequently, the set of irreducible oversemigroups of $S$ with multiplicity $3$ is:
$$
\{\langle 3, 4, 5\rangle, \langle 3, 4 \rangle, \langle 3, 5, 7 \rangle, \langle 3, 5 \rangle, \langle 3, 7 \rangle\}
$$
\end{ex}

 As a direct consequence of Theorem \ref{theo:12} and the identification between Kunz-coordinates vectors and numerical semigroups, we are able to describe minimal decompositions into $3$-irreducible numerical
 semigroups for a numerical semigroup with multiplicity $3$.

\begin{cor}
\label{cor:13}
Let $S$ be a numerical semigroup with multiplicity $3$ and $(x_1, x_2) \in \N^2$ its Kunz-coordinates vector. Then, either $S$ is irreducible or can be (minimally) decomposed into $3$-irreducible numerical semigroups as:
$$
S= \left\{\begin{array}{rl}
\langle 3, \frac{3x_1+1}{2} \rangle \cap \langle 3, \frac{3x_2+2}{2}\rangle  & \mbox{ if $x_1$ is odd and $x_2$ is even,}\\
\langle 3,  \frac{3x_1+1}{2} \rangle \cap \langle 3, \frac{3x_2+5}{2}, 3x_2+2 \rangle & \mbox{ if $x_1, x_2$ are odd,}\\
\langle 3,  3x_1+1, \frac{3x_1+4}{2} \rangle \cap \langle 3, \frac{3x_2+2}{2} \rangle & \mbox{ if $x_1, x_2$ are even,}\\
\langle 3,  3x_1+1, \frac{3x_1+4}{2} \rangle \cap \langle 3, \frac{3x_2+5}{2}, 3x_2+2 \rangle & \mbox{ if $x_1$ is even and $x_2$ is odd,}
\end{array}\right.
$$
\end{cor}

\begin{cor}
\label{cor:14}
Let $S$ be a numerical semigroup with multiplicity $3$. The decompositions of $S$ into irreducible numerical semigroups given in Corollary \ref{cor:13} are unique.
\end{cor}
\begin{proof}
The proof follows by noting that $\SG_3(S)=\SG(S)$ when $S \neq \{0, 3, \rightarrow\}$ and then, if a numerical semigroup is
not irreducible (being then $\#\SG(S)=2$) a decomposition of $S$ into irreducible numerical semigroups must consist
of two numerical semigroups with Frobenius number each of the special gaps in $\SG_3(S)$. Furthermore, since those semigroups
must be irreducible, its genus is also fixed. By Corollary 4 in \cite{rosales05}, there is only one numerical semigroup with fixed
genus and Frobenius number. Hence, the decomposition is unique.
\end{proof}
\begin{ex}
\label{ex:15}
Let $S=\langle 3, 23, 40 \rangle$. $S$ is a numerical semigroup with multiplicity $3$ and $\Ap(S, 3)=\{0, 40, 23\}$. Hence, its Kunz-coordinates vector is $x=(\frac{40-1}{3}, \frac{23-2}{3})=(13, 7)$. Since $x$ does not verify any of the conditions of Theorem \ref{theo:10}, $S$ is not irreducible. Then, since $13$ and $7$ are both odd, the minimal decomposition of $S$ into irreducible numerical semigroups is
$$
S= \langle 3,  \frac{3\times 13 +1}{2} \rangle \cap \langle 3, \frac{3\times 7+5}{2}, 3\times 7 +2 \rangle = \langle 3,  20\rangle \cap \langle 3, 13, 23 \rangle
$$
\end{ex}
In \cite{ijac11} it is also defined the notion of $m$-symmetry and $m$-pseudosymmetry of a
  numerical semigroup with multiplicity $m$, extending the previous notions of symmetry
and pseudosymmetry (see \cite{springer}). A numerical semigroup, $S$, with multiplicity
$m$ is $m$\textit{-symmetric} if $S$ is $m$-irreducible and $\F(S)$ is odd. On the other hand, $S$ is $m$\textit{-pseudosymmetric} if $S$ is $m$-irreducible and $\F(S)$ is even.

Another well-known set of numerical semigroups is the one of numerical semigroups that can be decomposed into $m$-symmetric numerical semigroups  (ISYM-semigroups). For the case when $m=3$, clearly, if $S$ is a numerical semigroup with multiplicity $3$ we can distinguish two cases: when $S$ is $3$-irreducible numerical semigroup or when it is not. In the first case, $S$ is a ISYM-semigroup if $S$ is $3$-symmetric, in the second case, $S$ is an ISYM-semigroup if the two $3$-irreducible numerical semigroups in the decomposition (Corollary \ref{cor:13}) are $3$-symmetric. Since the Frobenius numbers of both $3$-irreducible numerical semigroups are known, the elements in $\SG_3(S)$, $S$ is a ISYM-semigroup if all the elements in $\SG_3(S)$ are odd. This results in the following corollaries.

\begin{cor}
\label{cor:16}
Let $S$ be a numerical semigroup with multiplicity $3$ and Kunz-coordinates vector $x=(x_1, x_2) \in \N^2$. Then, $S$ is decomposable as an intersection of symmetric numerical semigroups with multiplicity $3$ if and only if one of the following conditions holds:
\begin{enumerate}
\item $S$ is a $3$-symmetric numerical semigroup.
\item $S$ is not $3$-irreducible, $x_1$ is odd and $x_2$ is even.
\end{enumerate}
\end{cor}

An analogous treatment can be done two analyze those numerical semigroups with multiplicity $3$ that can be decomposed as an intersection of $3$-pseudosymmetric numerical semigroups.

\begin{cor}
\label{cor:17}
Let $S$ be a numerical semigroup with multiplicity $3$ and Kunz-coordinates vector $x=(x_1, x_2) \in \N^2$. Then, $S$ is decomposable as an intersection of pseudosymmetric numerical semigroups with multiplicity $3$ if and only if one of the following conditions holds:
\begin{enumerate}
\item $S$ is a $3$-pseudosymmetric numerical semigroup.
\item $S$ is not $3$-irreducible, $x_1$ is even and $x_2$ is odd.
\end{enumerate}
\end{cor}

\section{4-irreducible numerical semigroups}
\label{sec:3}
In this section we study the set of irreducible numerical semigroups with multiplicity four. In this case,
\begin{lemma}
\label{lemma:18}
Every $4$-irreducible numerical semigroup different from $\{0, 4, \rightarrow\}$ is irreducible.
\end{lemma}
\begin{proof}
Let $S$ be a $4$-irreducible numerical semigroup. By Lemma \ref{lemma:1} one of the following conditions holds:
\begin{enumerate}
\item $S=\{x \in \N: x\ge m, x\neq \F(S)\} \cup \{0\}$. In this case, either $S=\langle 4, 5, 6\rangle$ , $S=\langle 4, 5, 7\rangle$, or $S=\langle 4, 6, 7\rangle$. All of them are irreducible.
\item $S$ is an irreducible numerical semigroup. In this case we are done.
\end{enumerate}

The semigroup $\{0, 4, \rightarrow\}$ is $4$-irreducible but it is not irreducible since $3 = \g(S) \neq \left\lceil\frac{\F(S)+1}{2}\right\rceil = 2$.
\end{proof}

Although the complete set of $4$-irreducible numerical semigroups does not coincide with the set of irreducible numerical semigroups with multiplicity $4$ as in the case when the multiplicity is $3$, the difference is only one element, $\widehat{S}=\{0, 4, \rightarrow\}$, which is closed under decompositions, i.e., it only appears in its own decomposition into $4$-irreducible numerical semigroups. Then, through this section we assume that the numerical semigroups with multiplicity $4$ are different from $\widehat{S}$.

Note also that the above result, is not further true when the multiplicity is greater than $4$. For the case when $m=5$, $\langle 5,6,8,9 \rangle$ is $5$-irreducible but it is not irreducible.

In the following lemma we describe the set of special gaps larger than $4$ of a numerical semigroup with multiplicity $4$.
\begin{lemma}
\label{lemma:19}
Let $S$ be a numerical semigroup with multiplicity $4$ and Kunz-coordinates vector $x=(x_1, x_2, x_3) \in \N^3$. Then, the set of special gaps larger than $4$ is:
{\scriptsize $$
\SG_4(S) = \left\{\begin{array}{cl} \{\} & \mbox{if $x_1=x_2=x_3=1$},\\
\{4x_1-3\} & \mbox{ if $x_1+x_2 \geq x_3+1$, $2x_1\geq x_2+2$, and $x_2+x_3 \leq x_1-1$,}\\
\{4x_1-3\} & \mbox{ if $x_1+x_2 \geq x_3+1$, $2x_1\geq x_2+2$, $x_2=1$ and $2x_3 \leq x_2$,}\\
\{4x_1-3\} & \mbox{ if $x_1+x_2 \geq x_3+1$, $2x_1\geq x_2+2$, $x_2=1$ and $x_2+x_3 \leq x_1-1$,}\\
\{4x_1-3\} & \mbox{ if $x_1=2$, $x_2=x_3=1$,}\\
\{4x_2-2\} & \mbox{ if $x_1+x_2 \geq x_3+1$, $x_2+x_3\geq x_1$, $2x_1\leq x_2+1$ and $2x_3\leq x_2$,}\\
\{4x_3-1\} & \mbox{ if $x_2+x_3 \geq x_1$, $2x_3\geq x_2+1$, and $x_1+x_2 \leq x_3$,}\\
\{4x_1-3, 4x_2-2\} & \mbox{ if $x_1+x_2 \geq x_3+1$, $2x_1\geq x_2+2$, $x_2+x_3\geq x_1$, $2x_3\geq x_2+1$, $x_1\geq 2$, and $x_2\geq 2$}\\
\{4x_1-3, 4x_3-1\} & \mbox{ if $x_1+x_2 \geq x_3+1$, $2x_1\leq x_2+1$, $x_2+x_3\geq x_1$, $2x_3\geq x_2+1$, $x_1\geq 2$, and $x_2\geq 2$}\\
\{4x_1-3, 4x_3-1\} & \mbox{ if $x_2=1$, $x_3 \leq x_1$, $x_3 \geq x_1-1$, $x_1\geq 2$, and $x_3\geq 2$,}\\
\{4x_2-2, 4x_3-1\} & \mbox{ if $x_1+x_2 \geq x_3+1$, $x_2+x_3\geq x_1$, $2x_3\geq x_2+1$, $2x_1\leq x_2+1$, $x_2\geq 2$, and $x_3\geq 2$}\\
\{4x_1-3, 4x_2-2, 4x_3-1\} & \mbox{ if $x_1+x_2 \geq x_3+1$, $2x_1\geq x_2+2$, $x_2+x_3\geq x_1$, $2x_3\geq x_2+1$, $x_2 \geq 2$, and $x_3 \geq 2$.}
\end{array}\right.
$$}
\end{lemma}
\begin{proof}
The result follows by applying  \eqref{sg} to compute the set $\SG_4(S)$ in terms of its Kunz-coordinates vector.
\end{proof}

Also, analogously to the case with multiplicity $3$ we have the following result concerning the irreducibility of a numerical semigroup with multiplicity $4$.
\begin{cor}
\label{cor:20}
Let $S$ be a numerical semigroup with multiplicity $4$ and $\Ap(S, 4)=\{0, 4x_1+1, 4x_2+2, 4x_3+3\}$. Then, $S$ is $4$-irreducible if an only if one of the following conditions holds:
\begin{enumerate}
\item $x_1=x_2=x_3=1$
\item $x_1+x_2 \geq x_3+1$, $2x_1\geq x_2+2$, and $x_2+x_3 \leq x_1-1$,
\item $x_1+x_2 \geq x_3+1$, $2x_1\geq x_2+2$, $x_2=1$ and $2x_3 \leq x_2$,
\item  $x_1+x_2 \geq x_3+1$, $2x_1\geq x_2+2$, $x_2=1$ and $x_2+x_3 \leq x_1-1$,
\item $x_1=2$, $x_2=x_3=1$,
\item $x_1+x_2 \geq x_3+1$, $x_2+x_3\geq x_1$, $2x_1\leq x_2+1$ and $2x_3\leq x_2$,
\item $x_2+x_3 \geq x_1$, $2x_3\geq x_2+1$, and $x_1+x_2 \leq x_3$.
\end{enumerate}
\end{cor}

In the following example we analyze irreducible numerical semigroups with multiplicity four that are generated by generalized arithmetic sequences by applying the above corollary.

\begin{ex}
\label{ex:27}
For $h, d$ positive integers, $k \in \{1,2,3\}$  and such that $\gcd(d, 4)=1$,  $S=\langle 4, 4h+d, 4h+2d,\ldots, 4+kd\rangle$
is said that generated by a generalized arithmetic sequence. In \cite{ramirezalfonsin} it is proved that
$\Ap(S, 4)=\{0, 4h\left\lceil\dfrac{1}{k}\right\rceil + d, 4h\left\lceil\dfrac{2}{k}\right\rceil + 2d, 4h\left\lceil\dfrac{3}{k}\right\rceil + 3d\}$.
Distinguishing the possible values of $k$ and $d \pmod 4$, the Kunz-coordinates vector of $S$ are:
$$
\left\{ \begin{array}{rl}
         (h+D, 2h+2D, 3h+3D) & \mbox{if $k=1$ and $d \equiv 1 \pmod 4$,}\\
         (3h+3D+2, 2h+2D+1, h+D) & \mbox{if $k=1$ and $d \equiv 3 \pmod 4$,}\\
         (h+D, h+2D, 2h+3D) & \mbox{if $k=2$ and $d \equiv 1 \pmod 4$,}\\
         (2h+3D+2, h+2D+1, h+D) & \mbox{if $k=2$ and $d \equiv 3 \pmod 4$,}\\
         (h+D, h+2D, h+3D) & \mbox{if $k=3$ and $d \equiv 1 \pmod 4$,}\\
         (h+3D+2, h+2D+1, h+D) & \mbox{if $k=3$ and $d \equiv 3 \pmod 4$,}\\
        \end{array}\right.
$$
where $D = \frac{d - d \pmod 4}{4}$.

Corollary \ref{cor:20} allows us to decide the irreducibility of those numerical semigroups just by checking in some inequalities are satisfied by the above Kunz-coordinates vectors.
By writting down the inequalities, it is easy to check that when $k=1$ or $k=2$, $S$ is always irreducible, while for $k=3$ the inequalities are never hold so $S$ is not irreducible.
 Furthermore, by Selmer's formulas $\F(S)=4(3h-1) + 3d$ if $k=1$ and $\F(S)=4(2h-1) + 3d$ if $k=2$. Since $d$, in these cases $\F(S)$ is always odd, being then $S$ symmetric.

This result was also proved by Matthews in \cite{mathews04}.
\end{ex}

In case the numerical semigroup is not $4$-irreducible, that is, in one of the cases of Corollary \ref{cor:20}, we can find a decomposition
 into $4$-irreducible numerical semigroups with at least two $4$-irreducible numerical semigroups. If $\#\SG_4(S)=2$, then a minimal decomposition will be given as an intersection of two $4$-irreducible numerical semigroups while if $\#\SG_4(S)=3$, such a decomposition may consists of two or three $4$-irreducible numerical semigroups.

First we analyze the case when the number of special gaps larger than $4$ is two. Then, it is enough to look for two $4$-irreducible numerical semigroups such that each of then has as Frobenius number each of the special gaps.

When the cardinality of $\SG_4(S)$ is two, to decompose $S$ into $4$-irreducible numerical semigroups we have to search for irreducible oversemigroups of $S$ with Frobenius number each of the two special gaps of $S$. 

\begin{theo}
\label{theo:21}
Let $S$ a numerical semigroup with multiplicity $4$ with Kunz-coordinates vector $x = (x_1, x_2, x_3) \in \N^3\backslash\{(1,1,1)\}$, and let $S'$ be an irreducible oversemigrop of $S$ with special gap $h_i=4(x_i-1)+i$. Then, the Kunz-coordinates vector of $S'$ are in the form $x-y \in \N^3$, with $y \in \N$ and such that:
$$
\begin{array}{rl}
y_1=0, y_2 \in [ x_2-x_1, x_2-1] \mbox{ and } y_3=-x_1+x_2+x_3+1-y_2 & \mbox{if i=1}\\
y_1  \in [x_1-x_3+1,x_1-\frac{x_3}{3}], y_2=0, y_3=x_1-x_2+x_3-y_1 & \mbox{if i=3}\\
y_1\in [x_1-x_3+1,x_1-\frac{x_3}{3}], y_2=x_1+x_2+x_3-y_1, y_3=0 & \mbox{if i=3}
\end{array}
$$
\end{theo}
\begin{proof}
Any oversemigroup of $S$ has Kunz-coordinates vector in the form $x^k=x-y^k$, where $y^k \in \N^3$ verifies the inequalities in \eqref{polytope:x} for $m=4$, that is
\begin{align*}
2y_1^k - y_2^k &\leq 2x_1-x_2\\
y_1^k + y_2^k - y_3^k &\leq x_1+x_2-x_3\\
y_2^k + y_3^k - y_1^k& \leq x_1+x_3-x_1+1\\
2y_3^k - y_2 &\leq 2x_3-x_2 +1
\end{align*}
and such that those oversemigroups are irreducible with $\F(x^k)=4(x_k-1)+k$, that is $y^k_k =0$ and
$$
y_1^k+y_2^k+y_3^k = x_1+x_2+x_3 - \left\lceil \frac{4(x_k-1)+k+1}{2} \right\rceil
$$
In what follows we analyze each $y^k$, for $k=1,2,3$:
\begin{enumerate}
\item For $k=1$ the conditions are, by fixing $y_1^1=0$ and $y_3^1 = -x_1+x_2+x_3+1-y_2^1$:
    \begin{align}
y^1_2 &\leq x_2-1\label{1.1}\\
y^1_2 &\geq -x_1+x_2+2\label{1.2}\\
y^1_2 &\geq x_2-2x_1\label{1.3}\\
y^1_2 &\leq x_2\label{1.4}\\
y^1_2 &\geq x_2 - \frac{2x_1+1}{2}\label{1.5}\\
y^1_2 \in \N\label{1.6}
\end{align}

Moreover, constraint \eqref{1.3} is redundant with constraint \eqref{1.5}, and constraint \eqref{1.2} is redundant when imposing \eqref{1.1} and \eqref{1.5}, and also \eqref{1.4} by \eqref{1.1}, so finally the lattice is
$$
y_2 \in [ x_2-\frac{2x_1+1}{2}, x_2-1] \cap \N = [ x_2-x_1, x_2-1] \cap \N
$$

\item For $i=2$,  $y^2_2=0$, and $y_3^2 = x_1-x_2+x_3-y_1^2$, so all the constraints can be written in terms of $y_1^2$. Then, the conditions are:
    \begin{align}
y^2_1 &\leq x_1-1\label{2.1}\\
y^2_1 &\geq x_1-x_2+1\label{2.2}\\
y^2_1 &\leq x_1-\frac{x_2}{2}\label{2.3}\\
y^2_1 &\leq x_1\label{2.4}\\
y^2_1 &\geq x_1 - x_2\label{2.5}\\
y^2_1 &\geq x_1 - \frac{x_2+1}{2}\label{2.6}\\
y^2_1 \in \N\label{2.7}
\end{align}

By an analogous discarding procedure, the above constraints can be written as the integer points inside the interval $[x_1-\frac{x_2+1}{2}, x_1-\frac{x_2}{2}]$.
Then, $h_1$ is in $\G(x-y)$ if and only if $y_1=0$. We can choose $y_1=0$ in the above system if $0 \geq x_1-\frac{x_2+1}{2}$
 (note that $0\geq x_1-\frac{x_2}{2}$ is always true by the conditions of being a Kunz-coordinates vector). Then, the condition is $2x_1 \geq x_2+1$. On the other hand, $h_3 \in \G(x-y)$ if $y_3=0$ is a eligible choice, that is, $x_1-x_2+x_3-y_2 = 0$ is a solution.
This is equivalent to $x_1-x_2+x_3 \leq  x_1-\frac{x_2}{2}$, that is, to $x_2 \geq 2x_3$.

\item Finally, for $h_3=4(x_3-1)+3$, the constraints can be written, by fixing $y_3=0$ and $y_2=x_1+x_2-x_3-y_1$, as:
    \begin{align}
y_1 &\leq x_1-1\label{3.1}\\
y_1 &\geq x_1-x_3+1\label{3.2}\\
y_1 &\leq x_1-\frac{x_3}{3}\label{3.3}\\
y_1 &\geq x_1-x_3\label{3.4}\\
y_1 &\leq x_1+x_3+1\label{3.5}\\
y_1 \in \N\label{3.6}
\end{align}

Whose set of solutions for $y_1$ is $[x_1-x_3+1,x_1-\frac{x_3}{3}] \cap \N$. Then $h_1 \in \G(x-y)$ if and only
 if $0 \geq x_1-x_3 +1$, that is, when $x_3 \geq x_1 + 1$. And $h_2 \in \G(x-y)$ if and only if $x_1+x_2-x_3 \leq x_1-\frac{x_3}{3}$, that is, when $2x_3 \geq 3x_2$.
\end{enumerate}
\end{proof}

As a consequence of the above theorem, minimal decomposition into irreducible numerical semigroups with multiplicity $4$ can be described when the number of special gaps larger than $4$ is two.
\begin{cor}
\label{cor:22}
Let $S$ be a numerical semigroup with multiplicity $4$ and Kunz-coordinates vector $x=(x_1, x_2, x_3) \in \N^3\backslash\{(1,1,1)\}$. Then, if $\SG_4(S)=\{4(x_i-1)+i, 4(x_j-1)+j\}$, a minimal decomposition of $S$ into irreducible numerical semigroups with multiplicity four is
{\small $$
S = \left\{\begin{array}{rl}
\langle 4, 4x_1+ 1, 4(x_2-y_2)+2, 4(x_1-x_2-1+y_2) + 3 \rangle \cap \langle 4, 4(x_1-y_1)+1, 4x_2+2, 4(x_2-x_1+y_1)+3 \rangle & \mbox{ if $i=1$, $j=2$}\\
\langle 4, 4x_1+ 1, 4(x_2-y_2)+2, 4(x_1-x_2-1+y_2) + 3 \rangle \cap \langle 4, 4(x_1-y_3)+1, 4(x_3-x_1+y_3)+2, 4x_3+3 \rangle & \mbox{ if $i=1$, $j=3$}\\
\langle 4, 4(x_1-y_1)+1, 4x_2+2, 4(x_2-x_1+y_1)+3 \rangle \cap \langle 4, 4(x_1-y_3)+1, 4(x_3-x_1+y_3)+2, 4x_3+3 \rangle & \mbox{ if $i=2$, $j=3$}
\end{array}\right.
$$}
with $y_1  \in [x_1-x_3+1,x_1-\frac{x_3}{3}]\cap \N$, $y_2 \in [ x_2-x_1, x_2-1]\cap \N$ and $y_3\in [x_1-x_3+1,x_1-\frac{x_3}{3}] \cap \N$.
\end{cor}
\begin{proof}
It follows directly by applying Theorem \ref{theo:21} and Lemma \ref{lemma:7}.
\end{proof}

\begin{ex}
\label{ex:22}
Let $S = \langle 4, 31, 53 \rangle$. $S$ is a numerical semigroup with multiplicity $4$ and $\Ap(S,4)=\{0,53,62,31 \}$. Then, its Kunz-coordinates vector is $x=(\frac{53-1}{4},\frac{62-2}{4}, \frac{31-3}{4}) = (13, 15, 7)$. By Lemma \ref{lemma:19}, the set $\SG_4(S)=\{4\times 13 - 3, 4\times 15 - 2\} = \{49, 58\}$. By Theorem \ref{theo:21}, the decompositions of $S$ into irreducible numerical semigroups with multiplicity $4$ are in the form:
\begin{align*}
S &=\langle 4, 53, 62 - 4y_2, -9 + 4y_2 \rangle \cap \langle 4, 53 - 4y_1,  62, 11 + 4y_1 \rangle
\end{align*}
with $y_1 \in [7, 10] \cap \N$ and $y_2 \in [2, 14]$. For example, taking $y_1=8$ and $y_2=6$, a minimal decomposition of $S$ into irreducible numerical semigroups is given by:
$$
S= \langle 4, 53, 38, 15 \rangle \cap \langle 4, 21, 62, 43 \rangle = \langle 4, 15 \rangle \cap \langle 4, 21, 43 \rangle
$$
\end{ex}

A direct consequence of Theorem \ref{theo:21} is the following result that states the number of minimal decompositions into irreducible numerical semigroups for a numerical semigroup with multiplicity $4$.
\begin{cor}
\label{cor:23}
Let $S$ be a numerical semigroup with multiplicity $4$, with Kunz-coordinates vector $x=(x_1, x_2, x_3) \in \N$ and $\SG_4(S)=\{4(x_i-1)+i, 4(x_j-1)+j\}$. Then, the number of minimal decompositions of
$S$ into irreducible numerical semigroups with multiplicity $4$ is the following:
$$
\left\{\begin{array}{rl}
        \left\lfloor \frac{2x_3}{3} \right\rfloor x_1 & \mbox{if $i=1$ and $j=2$}\\
	\left\lfloor \frac{2x_3}{3} \right\rfloor x_1	    & \mbox{if $i=1$ and $j=3$}\\
	\left\lfloor \frac{2x_3}{3} \right\rfloor^2 & \mbox{if $i=2$ and $j=3$}\\
       \end{array}\right.
$$
where $\lfloor q \rfloor$ is the floor part for any $q \in \Q$.
\end{cor}
\begin{proof}
 The result follows by counting the integer points in $[x_1-x_3+1,x_1-\frac{x_3}{3}]$,
 $[ x_2-x_1, x_2-1]$ and $[x_1-x_3+1,x_1-\frac{x_3}{3}]$, which are the interval where the $y$-variables take values in Theorem \ref{theo:21}.
\end{proof}

\begin{ex}
 \label{ex:24}
For $S= \langle 4, 31, 53 \rangle$ in Example \ref{ex:22}, since the Kunz-coordinates vector of $S$ is $(13,15,7)$ and $\SG_4=\{49, 58\}$, the number of minimal decomposituion in $S$ into irreducible
numerical semigroups with multiplicity $4$ is $\left\lfloor \frac{2\times 7}{3} \right\rfloor \times 13 = 4 \times 13 = 52$, which is the number of integer points inside $[7,10] \times [2, 14]$.
\end{ex}

In what follows, we analyze those numerical semigroups with multiplicity $4$ with three special gaps greater than $4$.
In this case one could find that the number of numerical semigroups involved in a minimal decomposition into $4$-irreducible numerical semigroups is $2$ or $3$.
The following result states when such a numerical semigroup is decomposable into $2$ irreducible numerical semigroups with multiplicity $4$.
\begin{theo}
\label{theo:25}
Let $S$ be a numerical semigroup with multiplicity $4$, with Kunz-coordinates vector $x=(x_1,x_2,x_3) \in \N^3$ and $\#\SG_4(S)=3$. Then, $S$ can be decomposed into two $4$-irreducible numerical semigroups if and only if one of the following conditions holds:
\begin{enumerate}
\item $x_2\leq x_1$, and $S=\langle 4, 4x_1+ 1, 4x_2+2, 4(x_1-x_2-1) + 3 \rangle \cap \langle 4, 4(x_1-y_3)+1, 4(x_3-x_1+y_3)+2, 4x_3+3 \rangle$,
\item $x_1 \geq x_3+2$, and $S=\langle 4, 4x_1+ 1, 4(x_1-x_3-1)+2, 4x_3 + 3 \rangle \cap \langle 4, 4(x_1-y_1)+1, 4x_2+2, 4(x_2-x_1+y_1)+3 \rangle$,
\item $2x_1\leq x_2+1$, and $S=\langle 4, 4x_1+1, 4x_2+2, 4(x_2-x_1)+3 \rangle \cap \langle 4, 4(x_1-y_3)+1, 4(x_3-x_1+y_3)+2, 4x_3+3 \rangle$,
\item $2x_3 \leq x_2$, and $S=\langle 4, 4x_1+ 1, 4(x_2-y_2)+2, 4(x_1-x_2-1+y_2) + 3 \rangle \cap \langle 4, 4(x_1-y_1)+1, 4x_2+2, 4(x_2-x_1+y_1)+3 \rangle$,
\item $x_3 \geq x_1+1$, and $S=\langle 4, 4(x_1-y_1)+1, 4x_2+2, 4(x_2-x_1+y_1)+3 \rangle \cap \langle 4, 4(x_1-y_3)+1, 4(x_3-x_1+y_3)+2, 4x_3+3 \rangle $,
\item $3x_2\leq 2x_3$, and $S=\langle 4, 4x_1+ 1, 4(x_2-y_2)+2, 4(x_1-x_2-1+y_2) + 3 \rangle \cap \langle 4, 4(x_1-y_3)+1, 4(x_3-x_1+y_3)+2, 4x_3+3 \rangle$.
\end{enumerate}
\end{theo}

\begin{proof}
Since $\#SG_4(S)=3$, $\SG_4(S)=\{4(x_1-1)+1, 4(x_2-2)+2, 4(x_3-1)+3\}$. Let us analyze each one of those special gaps:
\begin{enumerate}
\item $h_1=4(x_1-1)+1$.  Then, the semigroup $x-y$ whose Frobenius number is $h_1$
 covers also $h_2$ if $0$ is in $[ x_2-x_1, x_2-1]$, and then $y_2=0$ can be chosen, that is when $0 \geq  x_2-x_1$, or equivalently, when $x_2 \leq x_1$. Then, a minimal decomposition
is that of choosing $y_2=0$, and when $h_1$ and $h_3$ must be covered in Corollary \ref{cor:22}.
However, $x-y$ may not covers $h_2$ but can cover $h_3$, which is equivalent to $y_3=0$, that is possible when $y_3 = -x_1+x_2+x_3+1-y_2 = 0$ can be chosen, or equivalently, when $-x_1+x_2+x_3+1 \in  [ x_2-x_1, x_2-1] \cap \N$, which is the same that
$x_1 \geq x_3+1$. The decomposition is given by fixing $y_3=0$ (or equivalently $y_2 = -x_1+x_2+x_3+1$) in Corollary \ref{cor:22} when $h_1$ and $h_2$ are the special gaps.

\item $h_2=4(x_2-1)+2$. $h_1$ is in $\G(x-y)$ if and only if $y_1=0$. We can choose $y_1=0$ in the interval in Theorem \ref{theo:21} if $0 \geq x_1-\frac{x_2+1}{2}$
 (note that $0\geq x_1-\frac{x_2}{2}$ is always true by the conditions of being a Kunz-coordinates vector). The minimal decomposition follows by fixing $y_1=0$ in Corollary \ref{cor:22} when $h_2$ and $h_3$ are the special gaps.
 Then, the condition is $2x_1 \geq x_2+1$. On the other hand, $h_3 \in \G(x-y)$ if $y_3=0$ is a eligible choice, that is, $x_1-x_2+x_3-y_1 = 0$ is a solution.
This is equivalent to $x_1-x_2+x_3 \leq  x_1-\frac{x_2}{2}$, that is, to $x_2 \geq 2x_3$. Then, the decomposition is obtained by applying Corollary \ref{cor:22} when $y_1=x_1-x_2+x_3$ in the case when $h_1$ and $h_2$ are the special gaps to be covered.

\item Finally, for $h_3=4(x_3-1)+3$.  $h_1 \in \G(x-y)$ if and only
 if, again by Theorem \ref{theo:21}, $0 \geq x_1-x_3 +1$, that is, when $x_3 \geq x_1 + 1$. The decomposition is again obtained by aplying Corollary \ref{cor:22}
 when $h_3$ and $h_2$ are the special gaps. And $h_2 \in \G(x-y)$ if and only
 if $x_1+x_2-x_3 \leq x_1-\frac{x_3}{3}$, that is, when $2x_3 \geq 3x_2$. By Corollary \ref{cor:22}, the minimal decomposition coincided with the one when $y_1=x_1+x_3-x_3$ in the case when
$h_1$ and $h_3$ are the special gaps.
\end{enumerate}
\end{proof}

If none of the conditions of Theorem \ref{theo:25} holds, a minimal decomposition into irreducible numerical semigroups with multiplicity $4$ consists of the intersection of three
$4$-irreducible numerical semigroups. Those three semigroups can be described by choosing Kunz-coordinates vectors such that each of them has as Frobenius number each of the different three
special gaps. These choices are described as the integer points inside the above intervals. We summarize in the following result such a methodology to construct the decomposition.

\begin{theo}
 \label{theo:26}
Let $S$ be a numerical semigroups with multiplicity $4$, Kunz-coordinates vector $x=(x_1, x_2, x_3) \in \N^3$, and such that $\#\SG_4(S)=3$. Then, if none of the conditions of Theorem \ref{theo:25} holds, the  minimal decompositions of  $S$ into irreducible numerical semigroups with multiplicity $4$ is the following:
$$
S= \begin{array}{c}\langle 4, 4(x_1-y_1) +1, 4x_2+ 2, 4(x_2-x_1+y_1)+3 \rangle\\
\cap\\
\langle 4, 4x_1 +1, 4(x_2-y_2)+ 2, 4(x_1-x_2-1+y_2)+3 \rangle\\
\cap\\
\langle 4, 4(x_3-x_2+y_2) +1, 4(x_2-y_3)+ 2, 4x_3+3 \rangle
\end{array}$$
where $y_1 \in [x_1-\frac{x_2+1}{2}, x_1-\frac{x_2}{2}] \cap \N$, $y_2 \in [ x_2-x_1, x_2-1] \cap \N$ and $y_3 \in [x_1-x_3+1,x_1-\frac{x_3}{3}] \cap \N$.
\end{theo}

The following example illustrates the usage of the above result.
\begin{ex}
Let $S = \langle 4,21,18,23 \rangle$. $S$ is a numerical semigroup with multiplicity $4$ and its Kunz-coordinates vector is $x=(5,4,5)$. The set of special gaps greater than four of $S$ is $\SG_4(S)=\{17, 14, 19\}$. It is easy to check that $x$ does not verify any of the conditions of Theorem
\ref{theo:25}, and then, a minimal decomposition of $S$ into $4$-irreducible numerical semigroups involves $3$ semigroups. To compute one of those minimal decompositions, we apply Theorem \ref{theo:26} that
gives us directly the decomposition as:
$$
S= \langle 4, 21-y_1, 18, 4y_1-1 \rangle \cap \langle 4, 21, 18 - 4y_2, 4y_2+3 \rangle \cap \langle 4, 5+4y_2, 18-4y_3, 23 \rangle
$$
for any $y_1 \in [3,3] \in \N$, $y_2 \in [ 0, 3] \cap \N$ and $y_3 \in [1,3] \cap \N$. For instance, for $y_1=3$, $y_2=2$ and $y_3=3$ we have the decomposition
$$
S= \langle 4, 18, 18, 11 \rangle \cap \langle 4, 21, 10, 11\rangle \cap \langle 4, 13, 6, 23 \rangle = \langle 4, 11, 18 \rangle \cap \langle 4, 10, 11\rangle \cap \langle 4, 6, 13\rangle
$$
\end{ex}

Finally, we analyze with our approach the $4$-symmetry and $4$-pseudosymmetry of a numerical semigroup with multiplicity $4$. Let
$S \neq \{0, 4 , \rightarrow\}$ be a numerical semigroup with $\m(S)=4$ and Kunz-coordinates vector
$x=(x_1, x_2, x_3) \in \N^3$. If $S$ is irreducible, then, $\SG_4=\{4(x_i-1)+i\}$ for some $i \in \{1,2,3\}$, so $\F(S)=4(x_i-1)+i$.
 Hence, $S$ is symmetric if and only if $i$ is even, or equivalently, if $i =1, 3$, and $S$ is pseudosymmetric if and only if $i=2$.

If $S$ is not irreducible and $\SG_4(S)=\{4(x_i-1)+i, 4(x_j-1)+j\}$, $S$ is decomposable into symmetric numerical semigroups with
 multiplicity $4$ if and only $i=1$ and $j=3$, while $S$ is never decomposable as an intersection of pseudosymmetric numerical semigroups.
If $\#\SG_4(S)=3$, then, $S$ is decomposable into symmetric numerical semigroups if condition \emph{ 1.} or \emph{ 6.} in Theorem \ref{theo:25} is satisfied. Note that these conditions are hold if $S$ is decomposed into one numerical
semigroup with Frobenius number $4(x_1-1)+1$ and other with Frobenius number $4(x_3-3)+3$, covering some of them the even special
gap $4(x_2-1)+2$. Clearly, in this case $S$ can be never decomposed into pseudosymmetric numerical semigroups.


\begin{thebibliography}{99}

\bibitem{barucci97}  Barucci, V., Dobbs, D.E., and Fontana, M. (1997). \emph{Maximality Properties
in Numerical Semigroups and Applications to One-Dimensional Analytically
Irreducible Local Domains}. Memoirs of the Amer. Math. Soc. 598
.
\bibitem{barucci-froberg} Barucci, V. and Froberg, R. (1997). One-dimensional almost Gorenstein rings.
Journal of Algebra 188, p.418-442.
\bibitem{ijac11} Blanco, V. and Rosales, J.C. (2011). \emph{Irreducibility in the set of numerical semigroups with fixed multiplicity}. To appear in Internat. J. Algebra Comput.
\bibitem{siam11} Blanco, V. and Puerto, J. (2011). \emph{Integer programming for decomposing numerical semigroups into $m$-irreducible numerical semigroups}. Submitted. Available in \url{http://arxiv.org/abs/1101.4112}.
\bibitem{counting} Blanco, V.,  Garc\'ia-S\'anchez, P.A., and Puerto, J. (2011). \emph{Counting numerical semigroups with short generating functions}, To appear in Internat. J. Algebra Comput.

\bibitem{branco-nuno07} Branco, M.B. and Nuno, F. (2007). Study of algorithms for decomposition of a numerical semigroup. Int. J. Of Math. Models and
Methods in Applied Sciences 1 (2), 106--110.
\bibitem{delorme} Delorme, C. (1976). \emph{Sous-mono\"ides d'intersection compl\`ete de $\N$}. N. Ann. Sci. Ecole Norm. Sup.9, 145--154.

\bibitem{estrada94}  Estrada, M. and L\'opez, A. (1994). \emph{A note on symmetric semigroups and almost arithmetic sequences},
Comm. Algebra 22 (10) (1994), 3903--3905.

\bibitem{froberg87} Fr\"oberg, R., Gottlieb, C., and H\"aggkvist, R. (1987). \emph{On numerical semigroups}, Semigroup Forum 35, 63--83 (1987).
\bibitem{kunz1} Kunz, E. (1973). The value-semigroup of a one-dimensional Gorenstein ring, Proc.
Amer. Math. Soc. 25 (1973), 748--751.
\bibitem{kunz} Kunz, E. (1987). \"Uber dir Klassifikation numerischer Halbgruppen,
Regensburger matematische schriften 11.
    \bibitem{mathews04} Matthews, G.L. (2004). \emph{On numerical semigroups generated by generalized arithmetic sequences},
Comm. Algebra 32 (9), 3459--3469.



\bibitem{ramirezalfonsin} Ram\'irez-Alfons\'in, J.L. and R{\o}dseth, {\O}. J.  (2009). \emph{Numerical semigroups: Ap\'ery sets and Hilbert series}, Semigroup Forum, 79(2), 323--340.



    \bibitem{london02} Rosales, J. C., Garc\'ia-S\'anchez, P. A., Garc\'ia-Garc\'ia, J. I.,
and Branco, M. B. (2002). \emph{Systems of inequalities and numerical
semigroups}. J. London Math. Soc. (2) 65, no. 3, 611--623.

\bibitem{rosales02} Rosales, J.C. and Branco, M.B. (2002). \emph{Decomposition of a numerical semigroup as an intersection of irreducible numerical semigroups}. B. Belg. Math. Soc-Sim. 9 (2002), 373--381.
\bibitem{rosales02b} Rosales, J.C. and Branco, M.B., \emph{Numerical semigroups that can be expressed as an intersection of symmetric numerical semigroups}, J. Pure Appl. Algebra 171(2-3) (2002), 303--314.

\bibitem{rosales03} Rosales, J.C. and Branco, M.B. (2003). \emph{Irreducible numerical semigroups}. rosales03 J. Math. 209, 131--143.
\bibitem{rosales03b} Rosales, J.C. and Branco, M.B. (2003). \emph{Irreducible numerical semigroups with arbitrary multiplicity and embedding dimension}. J. Algebra 264 (2003), 305--315.

\bibitem{rosales04} Garc\'ia-S\'anchez, P.A. and Rosales, J.C. (2004). \emph{Every positive integer is the Frobenius number of an irreducible numerical semigroup with at most four generators}. Arkiv Mat. 42, 301--306.

\bibitem{rosales05} Rosales, J.C. (2005). Numerical semigroups with multiplicity three and four, Semigroup Forum 71 (2005), 323--331.

\bibitem{springer}
 Rosales, J.C. and Garc\'ia-Sanchez, P.A. (2009). Numerical semigroups, Springer, New York, NY, 2009. ISBN: 978-1-4419-0159-0.

     \bibitem{selmer77} Selmer, E.S. (1977). \emph{On a linear Diophantine problem of Frobenius}, J. Reine Angew. Math.
293/294, 1-17.

\end{thebibliography}
\end{document}